\documentclass[12pt, leqno]{article}

\usepackage{amsmath, amssymb, amscd, verbatim, xspace,amsthm}
\usepackage{latexsym, epsfig, color}

\newcommand{\A}{\ensuremath{{\mathbb{A}}}}

\newcommand{\Z}{\ensuremath{{\mathbb{Z}}}\xspace}
\renewcommand{\P}{\ensuremath{{\mathbb{P}}}}

\newcommand{\ra}{\rightarrow}
\newcommand{\lra}{\longrightarrow}

\newcommand\Sym{\operatorname{Sym}}

\newcommand\tensor{\otimes}
\newcommand\isom{\cong}

\newcommand\tesnor{\otimes}

\newcommand\disc{\operatorname{disc}}

\newcommand\GL{\operatorname{GL}}

\newcommand\SL{\operatorname{SL}}

\newcommand\Spec{\operatorname{Spec}}
\newcommand\Proj{\operatorname{Proj}}

\newcommand\ts{^{\tensor 2}}
\newcommand\wt{\wedge^2}

\newcommand\GZ{\ensuremath{\GL_2(\Z)}\xspace}
\newcommand\Sf{\ensuremath{{S_f}}\xspace}
\newcommand\OSf{\ensuremath{\mathcal{O}_\Sf}\xspace}

\newcommand\BS{\ensuremath{S}\xspace}

\newcommand\OS{\ensuremath{{\mathcal{O}_\BS}}\xspace}

\newcommand\Gto{\GZ\times\GL_1(\Z)}
\newcommand\Go{\GL_1(\Z)}

\newcommand\map[4]{\ensuremath{\begin{array}{ccc}#1&\lra&#2\\#3&\mapsto&#4\end{array}}}
\newcommand\bij[2]{\ensuremath{\left\{\parbox{2.5 in}{#1}\right\} \longleftrightarrow \left\{\parbox{2.5 in}{#2}\right\}}}

\newcommand\bq{\begin{equation}}
\newcommand\eq{\end{equation}}

\newtheorem{proposition}{Proposition}[section]
\newtheorem{theorem}[proposition]{Theorem}
\newtheorem{corollary}[proposition]{Corollary}

\theoremstyle{remark}
\newtheorem{remark}[proposition]{Remark}

 \usepackage{fullpage}

\newenvironment{notation}{\vspace{2 ex}{\noindent{\bf Notation. }}}{\vspace{2 ex}}
\newenvironment{keycon}{\vspace{2 ex}{\noindent{\bf Key Construction. }}}{\vspace{2 ex}}
\newenvironment{globcon}{\vspace{2 ex}{\noindent{\bf Global Construction. }}}{\vspace{2 ex}}

\title{Gauss composition over an arbitrary base}

\author{Melanie Matchett Wood\thanks{mwood@math.stanford.edu}\\
Stanford University, Department of Mathematics, \\
Building 380, Sloan Hall, \\
Stanford, California 94305}

\begin{document}

\maketitle

\begin{abstract}
The classical theorems relating integral binary quadratic forms and ideal classes of quadratic orders have been of tremendous importance in mathematics, and many authors have given extensions of these theorems to rings other than the integers.  However, such extensions have always included hypotheses on the rings, and the theorems involve only binary quadratic forms satisfying further hypotheses.  We give a complete statement of the relationship between binary quadratic forms and modules for quadratic algebras over any base ring, or in fact base scheme.  The result includes all binary quadratic forms, and commutes with base change.  We give global geometric  as well as local explicit descriptions of the relationship between forms and modules.
\end{abstract}


\section{Introduction}\label{S:intro}

	The classical theorems relating binary quadratic forms $ax^2+bxy+cy^2$ with $a,b,c\in\Z$ and
ideal classes in quadratic rings have seen wide application.
The usual statements of these theorems are as follows.
\begin{theorem}[e.g. see {\cite[Theorems 5.2.8 and 5.2.9]{Cohen}}]\label{classical}
 If $D$ is a negative integer, then there is a bijection between the $\SL_2(\Z)$ equivalence classes of positive definite, primitive, integral binary quadratic forms of discriminant $D$ and the elements of the class group of the quadratic order with discriminant $D$.

If $D$ is a positive, non-square integer, then there is a bijection between the $\SL_2(\Z)$ equivalence classes of primitive, integral binary quadratic forms of discriminant $D$ and the elements of the narrow class group of the quadratic order with discriminant $D$.
\end{theorem}
 
Theorem~\ref{classical} is the basis of the best computational algorithms for working with class groups of quadratic orders (see \cite[Chapter 5]{Cohen}), and also is the basis for genus theory, which gives a much more complete understand of the two part of quadratic class groups than we have of any other parts of any other class groups.  In this paper, we give the generalization
of Theorem~\ref{classical} when $\Z$ is replaced by an arbitrary base ring or scheme (Theorems~\ref{T:gen}, \ref{T:bqf}, and \ref{T:tbqf}).  We also state the theorems in a way that allows the removal of all hypothesis on the discriminant and the integral binary quadratic forms, and also unites the case of positive and negative discriminant.  For example, over $\Z$, we have the following.

\begin{theorem}\label{T:classGL2}
There is a bijection 
$$\bij{twisted \GZ-equivalence classes of non-degenerate binary quadratic forms}{isomorphism classes of $(C,I)$, with $C$ a non-degenerate oriented quadratic ring, and $I$ a full ideal class of $C$}.$$
This bijection is discriminant preserving, i.e. if $f\leftrightarrow(C,I)$ then $\disc f=\disc C$.
\end{theorem}

A \emph{quadratic ring} is a ring that is a free rank 2 \Z-module under addition, and orders in quadratic number fields are the typical (but not only) examples of quadratic rings.
An  ideal $I$ of a quadratic ring $C$ is \emph{full} if it is rank 2 as a \Z-module.  (When $C$ is a domain,
full is equivalent to non-zero.)
 A form or ring is \emph{non-degenerate} if its discriminant is non-zero.  
An \emph{oriented quadratic ring} is a quadratic ring $R$ with a choice of generator of $R/\Z$ (there are two choices),
and an isomorphism of oriented quadratic rings must preserve this generator.  
The most subtle issue in Theorem~\ref{T:classGL2} is the \GZ action.  An element $g=
\left(
\begin{smallmatrix}
k & \ell\\
m&n
\end{smallmatrix}\right)$ acts on a form $f(x,y)=ax^2+bxy+cy^2$ by
$$g \circ f = \frac{1}{\det g} f(kx+\ell y,mx+ny).$$
We will call this the \emph{twisted} \GZ action on binary quadratic forms.

Theorem~\ref{T:classGL2} gives a bijection between classes of binary quadratic forms and elements of the class group of real quadratic orders which is often overlooked when the classical theorem relating $\SL_2(\Z)$ classes of forms to the narrow class group is mentioned.  Theorem~\ref{T:classGL2} also doesn't restrict to positive definite or primitive forms.
In fact, we will later remove the non-degeneracy condition from Theorem~\ref{T:classGL2} (in Theorem~\ref{T:Zdeg} over \Z)

It is interesting to note that quadratic rings are in bijection with integers $D$ congruent to
0 or 1 modulo 4; the bijection is given by the discriminant of the quadratic ring.
The condition that the quadratic ring is \emph{oriented} may seem unnatural.  However, without requiring orientation, all quadratic rings have an automorphism, which in particular takes an ideal class to its inverse, and therefore the equivalence classes of $(C,I)$ would not correspond to elements of a class group but rather to elements of a quotient of the class group.  In fact, Theorem~\ref{T:classGL2} is a specialization of the following result which removes the orientation condition.

\begin{theorem}\label{T:classGL2GL1}
There is a bijection
$$\bij{$\Gto$ equivalence classes of non-degenerate binary quadratic forms}{isomorphism classes of $(C,I)$, with $C$ a non-degenerate quadratic ring and $I$ a full ideal class of $C$}.$$
This isomorphism is discriminant preserving.
\end{theorem}
In this theorem, $\Go$ acts on a form by $(k) \circ f = kax^2+kbxy+kcy^2$.
This allows multiplication of the form by $-1$.
  When we also act by
$\Go$, we can change the $\GZ$ action by any twist without changing the equivalence classes, and so here we
may as well consider the most natural (non-twisted) \GZ action  
$$g \circ f = f(ax+by,cx+dy).$$  

Theorems~\ref{T:classGL2} and \ref{T:classGL2GL1} are called \emph{Gauss composition} because they give a composition law on
binary quadratic forms with the same discriminant, worked out explicitly by Gauss \cite{Gauss}, but given now more simply by multiplication of ideal classes. 
Given this beautiful and extremely useful correspondence over the integers, one naturally wonders what happens when
the integers are replaced by other rings, or if one is inclined to think geometrically,
when the integers are replaced by the sheaf of functions on another scheme.
In this paper, we give the Gauss composition correspondence over an arbitrary base scheme $S$.  In the case
$S=\Spec R$, we have the correspondence over an arbitrary base ring $R$.  

We now give the definitions necessary to work over an arbitrary scheme. 
The most important change from $\Z$ to a scheme $S$ is that previously we allowed the quadratic form
to have variables $x$ and $y$ which generate a free rank 2 $\Z$-module and thus were acted on by $\GZ$.  Over an arbitrary 
scheme $S$, we allow the variables of the binary form to be in a locally free rank 2 $\OS$-module
instead of just a free module.  (Over \Z all locally free modules are free.)
A \emph{binary quadratic form} over $S$ is a locally free rank 2 $\OS$-module $V$
and a global section $f\in \Sym^2 V$. Isomorphisms of binary quadratic forms $(V,f)$ and $(V',f')$ are
given by isomorphisms $V\ra V'$ that send $f$ to $f'$.   
The notion of isomorphism classes of binary quadratic forms replaces $\GZ$-equivalence.
 Over $\Z$, isomorphism classes of binary quadratic forms correspond exactly to non-twisted $\GZ$ equivalence classes of binary
 quadratic forms.

Now, we take a different point of view on our above $\Go$ action.  We could consider a form $f=ax^2z+bxyz+cy^2z$ and then view our above
$\Go$ action as the invertible changes of coordinates in the $z$ variable (which over $\Z$ are just multiplication by $\pm1$).
Then analogously to our above transition to an arbitrary scheme $S$, we make the following definition.  A \emph{linear binary
quadratic form} over $S$ is a locally free rank $2$ $\OS$-module $V$, a locally free rank 1 $\OS$-module $L$, and a global section $f\in \Sym^2 V\tensor L$. Isomorphisms are given by isomorphisms $V \ra V'$ and $L \ra L'$ that send $f \mapsto f'$.   
Over $\Z$, isomorphism classes of linear binary quadratic forms correspond exactly to the $\Gto$ equivalence classes described 
in Theorem~\ref{T:classGL2GL1}. A \emph{quadratic algebra} $C$ over $S$ is a locally free rank 2 \OS-algebra, which by taking $\Spec$ corresponds exactly to finite, flat degree 2 $S$-scheme.  We now specify which $C$-modules will appear in Gauss composition.
   A $C$-module $M$ is \emph{traceable} if $M$ is a locally free rank 2 $\OS$-module and if $C$ and $M$ give the same trace map
$C \ra \OS$.  
We now are ready to give the main theorem of Gauss composition over an arbitrary base.

\begin{theorem}\label{T:gen}
There is a bijection
$$
 \left\{\parbox{2.5 in}{isomorphism classes of linear binary quadratic 
forms/$S$}\right\}\longleftrightarrow \left\{\parbox{2.5 in}{isomorphism classes of $(C,M)$, with $C$ a quadratic algebra$/S$, and
$M$ a traceable $C$-module }\right\}.$$
Given $(C,M)$ and a corresponding $f\in\Sym^2 V\tensor L$, we have $M\isom V $ as $\OS$-modules and
$$
C/\OS \isom \wedge^2 V^* \tesnor L^*
$$
as \OS-modules.  
\end{theorem}

In the case when $S=\Spec \Z$ 
and we 
consider
only non-degenerate objects,
 we recover the classical Theorem~\ref{T:classGL2GL1} (because, in this case, all traceable modules are realizeable as ideal classes).
We give a simple and concrete proof of Theorem~\ref{T:gen} in Section~\ref{S:proof}.  We also reinterpret
the proof in terms of moduli stacks.   Theorem \ref{T:gen} comes from an isomorphism of moduli stacks
parametrizing linear binary quadratic forms on the one hand and parametrizing quadratic algebras and their traceable modules
on the other.  The content of the ``isomorphism of moduli stacks'' result is that the bijection of Theorems 
\ref{T:gen} commutes with base change of the scheme $S$ or ring $R$ (when $S=\Spec R$).
Moreover, we give explicit descriptions of the bijection maps in terms of bases for the ring and module 
 (which describe the map locally on the base scheme) as well as global descriptions of the bijection, both
geometric and algebraic (see Section~\ref{S:global}).

There has been previous work done to generalize Gauss composition to an arbitrary ring (see \cite{Butts}, \cite{Dulin}, \cite{Kap}, \cite{Towber}), usually with conditions
on the base ring (for example, that it is a Bezout domain or that 2 is not a zero-divisor), with conditions
on the forms and modules (for example, free, primitive, irreducible, invertible), or with
orientations of the rings or modules.  In this paper we are able to give a complete theorem without
any such conditions.  The closest previous work is that of Kneser \cite{Kneser}, who works over
an arbitrary ring and gives a construction of quadratic algebras and modules from quadratic maps
using Clifford algebras.  Kneser mainly studies composition of two modules and does not formulate a theorem
in the style of this paper.  Lenstra has given a talk \cite{Lenstra} based on Kneser's work in
which he suggested theorems for primitive forms and invertible modules in the style of this paper.  In Section~\ref{S:Kneser},
we given a closer comparison of our terminology and results to those of Kneser.  Our work introduces definitions and a framework that
are used to study binary forms of degree $n$ in \cite{binarynic}, which gives bijection theorems in the style of
Theorem~\ref{T:gen} for binary forms of degree $n$.  The global geometric and algebraic constructions of this paper
can be generalized to binary forms of degree $n$ and are critical to the study of such forms in \cite{binarynic}.
Besides its independent interest, this paper is an introduction to and motivation for the results of \cite{binarynic}.

A linear binary quadratic form is \emph{primitive} if everywhere locally where $V$ and $L$ are free (and $x,y$ is a basis for $V$ and $z$ is a basis for $L$),
$f$ can be written as $ax^2z+bxyz+cy^2z$, where
$a,b,c$ generate the unit ideal in $\OS$.
 A form over $\Z$ is \emph{primitive} if its coefficients generate the unit ideal in $\Z$, and over $\Z$, primitive forms
correspond exactly to the invertible ideal classes, and we the the same phenomenon happens in general.

\begin{theorem}\label{T:prim}
We can restrict the bijection of Theorem~\ref{T:gen} to a bijection
$$
 \left\{\parbox{2.5 in}{isomorphism classes of primitive linear binary quadratic 
forms/$S$}\right\}\longleftrightarrow \left\{\parbox{2.5 in}{isomorphism classes of $(C,M)$, with $C$ a quadratic algebra$/S$, and
$M$ an invertible $C$-module }\right\}.$$
In particular, if $M$ is an invertible (i.e. locally free rank 1) $C$-module, then $M$ is traceable.

\end{theorem}

Note that if $M$ is \emph{locally on $S$} a free $C$-module of rank 1, then $M$
is clearly \emph{locally on $C$} a free $C$-module of rank 1.
It turns out that the converse is also true.  If $M$ is \emph{locally on $C$} a free $C$-module of rank 1, then $M$
is \emph{locally on $S$} a free $C$-module of rank 1.

We can also talk about discriminants of quadratic algebras and linear binary quadratic forms over $S$.
The discriminant of a quadratic $\OS$-algebra $C$ is a global section of
$(\wedge^2 C)^{\tensor -2} \isom (C/\OS)^{\tensor -2}$ given by the determinant of the trace map $C\tensor C\ra\OS$.
The discriminant of a linear binary quadratic form $f\in \Sym^2 V \tensor L$ is a
global section of $(\wedge^2 V \tensor L)\ts$ which is given locally where
 $V$ and $L$ are free by $ax^2z+bxyz+cy^2z$ has discriminant $(b^2-4ac)((x\wedge y)\tensor z)^2$.  
We can view the discriminant (in either case) as a pair $(N,d)$, where $N$ 
is a locally free $\OS$-module of rank 1 and $d\in N\ts$, with isomorphisms given by isomorphisms $N \isom N'$
sending $d$ to $d'$.

\begin{theorem}\label{T:disc}
 In the bijection of Theorem~\ref{T:gen},
the isomorphism $C/\OS \isom \wedge^2 V^* \tesnor L^*$ 
gives a map 
$
(C/\OS)^{\tensor -2} \isom  (\wedge^2 V \tesnor L)\ts
$ which maps the discriminant of $(C,M)$ to the discriminant of $f$.
In other words, the bijection of Theorem~\ref{T:gen} is discriminant preserving.
\end{theorem}

We can specialize Theorem~\ref{T:gen} by specifying the locally free rank 1 module $L$ (see Section~\ref{S:bqf}).
This allows us to give a correspondence for binary quadratic forms over an arbitrary scheme and modules of quadratic algebras.
As another specialization, we get a version of Theorem~\ref{T:classGL2} over an arbitrary base.

Theorem \ref{T:gen} parametrizes traceable modules for quadratic rings over some base, but
the original theorems of Gauss composition are about ideal classes.  We can compare traceable modules and ideal classes.
To do this we work over an integral domain or scheme $D$.  An ideal $I$ of a quadratic $D$-algebra $C$ is \emph{full} if it is a locally
free rank 2 $D$-module.  Two ideals $I$ and $I'$ are \emph{equivalent} if there are non-zero-divisors $c,c'\in C$ such that
$cI=c'I'$, and this equivalence defines ideal classes.
Over an integral domain or scheme, an object is \emph{degenerate} if its discriminant is zero.

\begin{proposition}\label{P:idmod}
 When $D$ is an integral domain or scheme, 
all full ideals of a quadratic algebra $C$ over $D$ are traceable $C$-modules.
When $C$ is non-degenerate, all
traceable $C$-modules are realized as full ideals of $C$.  Two full ideals of
$C$ are in the same ideal class if and only if they are isomorphic as modules.  
\end{proposition}

\begin{corollary}\label{C:ideals}
When $D$ is an integral domain or scheme, there is a discriminant preserving bijection
$$
 \left\{\parbox{2.5 in}{isomorphism classes of non-degenerate linear binary quadratic 
forms/$D$}\right\}\longleftrightarrow \left\{\parbox{2.5 in}{isomorphism classes of $(C,I)$, with $C$ a non-degenerate quadratic algebra$/D$, and
$I$ a full ideal class of $C$ }\right\}.$$
\end{corollary}

Note that we do not require $C$ to be an integral domain or scheme.  When $C$ is a degenerate quadratic $D$-algebra, there
are traceable modules which do not occur as ideals of $C$.  However, since these modules do correspond
to linear binary quadratic forms we see that traceable modules are naturally included in the most complete theorem.

\subsection{Outline of the paper}
In Section~\ref{S:proof}, we prove Theorems~\ref{T:gen}, \ref{T:prim}, and \ref{T:disc} and give 
a local explicit description of the bijection of Theorem~\ref{T:gen}.  In Section~\ref{S:global},
we give a global geometric construction of $(C,M)$ from a linear binary quadratic form and
a
global algebraic construction of a linear binary quadratic form from a pair $(C,M)$.  In Section~\ref{S:idmod},
we relate traceable modules to ideals in order to prove Proposition~\ref{P:idmod} and Corollary~\ref{C:ideals}.  We also give Theorem~\ref{T:gen} over $\Z$
and see that not all the modules in the theorem are realizable as ideals.
In Section~\ref{S:bqf}, we specialize Theorem~\ref{T:gen} to forms $f\in\Sym^2 V \tesnor L$ with a given $L$,
and recover a version of Theorem~\ref{T:classGL2} over an arbitrary scheme.
Finally, in Section~\ref{S:Kneser} we relate our terminology and results to those of Kneser \cite{Kneser},
who worked on Gauss composition over an arbitrary base ring.
 
\begin{notation}
Given an $\OS$-module $P$, we let
$P^*$ denote the $\OS$-module $\mathcal{H}om(P,\OS)$, even
when $P$ is also a module for some other $\OS$-algebra.
Given a sheaf $\mathcal{G}$ on $S$, we write
$x\in \mathcal{G}$ to denote that $x$ is a global section of $\mathcal{G}$.
We use $\Sym_k V$ to denote the submodule of $V^{\tensor k}$ that is fixed by the $S_k$ action, and
we use   $\Sym^k V$ to denote the usual quotient of $V^{\tensor k}$.

 Normally, in the language of algebra,
 one says that an $R$-module $M$ is locally free of rank $n$
if for all prime ideals $\wp$ of $R$, the localization $M_\wp$ is free of rank $n$; in the
geometric language we would describe this situation as ``$M$ is free in every stalk.''
However, if we have a scheme $S$ and an $\OS$-module $M$, we normally
say that $M$ is locally free of rank $n$ if on some open cover of $S$ it is free of rank $n$;
in the algebraic language this is equivalent to saying that for every prime ideal $\wp$ of $R$, 
there is an $f\in R \setminus \wp$ such that the
 localization $M_f$ is free of rank $n$.  It turns out that over a ring, 
the geometric condition of locally free of rank $n$ is equivalent to being finitely 
generated and having the algebraic condition of locally free of rank $n$ (\cite[II.5.3, Theorem 2]{B}).
In this paper, we use the geometric notion of locally free of rank $n$, even when working over a ring. 
\end{notation}

\section{Proof of Theorem~\ref{T:gen}}\label{S:proof}

We give a simple proof of Theorem~\ref{T:gen}.

\begin{keycon}
Given a linear binary quadratic form $f\in \Sym^2 V\tesnor L$, we we construct
$C$ and $M$ as $\OS$-modules as follows:
\begin{equation}\label{E:defC}
 C= \OS \oplus \wedge^2 V^* \tesnor L^* \qquad M =V.
\end{equation}
Next, we need to specific the algebra and $C$-module structure of $C$ and $M$ respectively.
First, consider the case then $V$ and $L$ are free such that
$V=\OS x\oplus\OS y$ and $L=\OS z$.  We rename $(1,0)$ 
and $(0,(x^* \wedge y^* )\tensor z^*)$
of $C$ to $1$ and $\tau$.  Suppose $f=ax^2z+bxyz+cy^2z$. 
Now we let $1$ be the multiplicative identity of $C$, and let the rest of
the algebra and module structures be as follows:
$$\tau^2=-b\tau-ac \qquad \qquad\tau x=-cy-bx\qquad\qquad\tau y=ax$$
This gives $M$ the structure of a traceable $C$-module. 
Also note that $\disc f=(b^2-4ac)((x\wedge y)\tensor z)^2=\disc C$.

For a general $f$, we need to specify the algebra and module structures of $C$ and $M$ by
giving $C\tesnor C \ra C$ and $C\tensor M \ra M$ satisfying the axioms of rings and modules.
Since the module of such homomorphisms is a sheaf, it suffices to give the algebra and module structures locally
when $V$ and $L$ are free, which is what we have done above.  To see that the local definitions agree on overlaps,
we just check that if we had chosen different bases for the free $V$ and $L$ that we would have gotten the
same algebra and module structure on $C$ and $M$.  This is a simple computation.  Also
note that $\disc f$ and $\disc C$ correspond in the isomorphism
$(\wedge^2 V \tensor L)\ts\isom (C/\OS)^{\tesnor -2}$ because they correspond locally.
\end{keycon}

Given a quadratic $\OS$-algebra $C$ and a traceable $C$-module $M$
we can construct $\OS$-modules $V=M$ and $L=\wedge^2 V^* \tesnor (C/\OS)^*$.
In the case that
$C$ and $M$ are free $\OS$-modules, we can choose bases for $1,\tau$, and $x,y$
for $C$ and $M$ respectively, such that 
$$\tau x=-cy-bx\quad\text{and}\quad\tau y=ax$$
for some $a,b,c\in\OS$.  (Shifting $\tau$ by an element of
$\OS$ if necessary, we can ensure that $\tau y$ is a multiple of $a$
and we call such a basis $1,\tau$ \emph{normalized}.)  If $\tau^2=-q\tau-r$ with $q,r\in\OS$, then
the traceability condition tells us that $q=b$ and the condition
that $\tau^2=-q\tau-r$ tells us that $r=ac$.  
From this $(C,M)$ we can construct a form
$ax^2z+bxyz+cy^2z$, where $z=x^* \wedge y^* \tensor \bar{\tau}^*$ 
and $\bar{\tau}$ is the image of $\tau$ in $C/\OS$.

Now for an arbitrary $(C,M)$, this construction specifies $f\in \Sym^2 V \tesnor L$
locally on $S$ where $C$ and $M$ are free $\OS$-modules.  To see that the local definitions of
$f$ agree on overlaps, we can do a simple computation to see that if we had chosen a different 
basis for $M$, and a different normalized basis for $C$, we would obtain the same form.

The constructions of the above two paragraphs are inverse to each other, because we 
see they are locally inverse by construction.  Thus we have proved the bijection of Theorem~\ref{T:gen},
as well as Theorem~\ref{T:disc} which says that Theorem~\ref{T:gen} is discriminant preserving.

\subsection{Primitive forms}
If we had a form such that $V$ and $L$ were free as above and $f=ax^2z+bxyz+cy^2z$
with $a,b,c$ generating $\OS$ as an $\OS$-module, then we also have
$a,a+b+c,c$ generating $\OS$ as an $\OS$-module.  We can cover $S$ by
subsets $D_a,D_{a+b+c},D_c$ on which $a,a+b+c,c$ are invertible respectively.  By changing
the basis of $V$ on $D_{a+b+c}$ and $D_c$ we can assume that $a$ is invertible in each open subset.
By changing the basis of $V$ again, we can assume that $a=1$.  When $a=1$, we see
that $M$, as given by the Key Construction, is a free rank 1 $C$-module.  

On the other hand, if $C$ is a free $\OS$-module and if $M$ is a free 
$C$-module of rank 1, then we can choose a $\OS$-module basis $x$ and $y$ of $M$
and a $\OS$-module basis $1,\tau$ of $C$ such that 
$$\tau^2=-b\tau-c, \quad \tau x=-cy-bx\quad\text{and}\quad\tau y=x$$
with $b,c\in\OS$.  Such a $(C,M)$ gives a form $x^2z+bxyz+cy^2z$, which is primitive.
So we see that in the bijection of Theorem~\ref{T:gen}, primitive forms correspond exactly to $(C,M)$
such that locally on $S$, we have $M$ a free $C$-module of rank 1.  Note that
any module which is locally on $S$ a free $C$-module of rank 1 is traceable.  
Thus we conclude that in the bijection of Theorem~\ref{T:gen}, primitive forms correspond to $M$ which are locally on $S$ free $C$-modules of rank 1, proving Theorem~\ref{T:prim}

\subsection{Moduli Stacks}
Another way to see the proof of Theorem~\ref{T:gen} is as follows.
\begin{theorem}\label{T:stacks}
 There is an equivalence of moduli stacks between the moduli stack of linear binary quadratic forms
and the moduli stack of pairs $(C,M)$ where $C$ is a quadratic algebra and $M$ is a traceable $C$-module.
\end{theorem}
\begin{proof}
 This is just another way of formalizing the above argument.  We can rigidify the two moduli problems of the theorem. 
 We have that $\A^3$ is the moduli scheme $\mathcal{F}_B$ of linear binary quadratic forms $(V,L,f)$ on $S$ with 
given $V\isom\OS^2$ and $L\isom \OS$.  
This is parametrizing linear binary forms with $V$ and $L$ free and with chosen bases.
Also, we saw above that
$\A^3$ is the moduli scheme $\mathcal{M}_B$ of
pairs $(C,M)$ where $C$ is a quadratic $\OS$-algebra, $M$ is a traceable $C$-module, and we have given isomorphisms
$C/\OS \isom \OS$ and $M \isom \OS^2$.  
This is parametrizing quadratic $\OS$-algebras $C$ with traceable modules $M$, with
$C/\OS$ and $M$ free and with chosen bases.
  An element $g\in \GL_2$ acts on 
the isomorphisms of $M$ and $V$ with $\OS^2$  by composing with $g$
and acts on the isomorphism $C/\OS \isom \OS$ by composing with $\det g^{-1}$.
An element $g\in \GL_1$ acts on the isomorphism $C/\OS \isom \OS$ by composing with $g^{-1}$.
This gives an action of the group scheme $\GL_2\times\GL_1$ on both of the rigidified moduli spaces.
  The Key Construction of $(C,M)$ from the universal form gives us
$\mathcal{F}_B\isom\mathcal{M}_B$ which is equivariant for the $\GL_2\times\GL_1$ actions.
Thus we have an equivalence of the quotient stacks by $\GL_2\times\GL_1$, which are the moduli stacks in this theorem.
\end{proof}

The bijection of Theorem~\ref{T:gen} is a corollary of this theorem by comparing $S$ points of the two moduli stacks.

\section{Global Descriptions of the Bijection}\label{S:global}
We have proven Theorem~\ref{T:gen}, and the
maps in the bijection are given locally in a simple and completely explicit form
by the Key Construction.
We wish now to give global descriptions of the maps in each direction in our bijections.

A linear binary quadratic form $f$ over $S$
defines a closed subscheme $S_f$ of $\P(V)$, where we use $\P(V)$ to denote $\Proj (\Sym V)$.
We let $\mathcal{O}(k)$ denote the usual sheaf on $\P(V)$ and
let $\OSf(k)$ denote the pullback of $\mathcal{O}(k)$ to $S_f$.
  The global functions of
$S_f$ give an $\OS$-algebra and the global sections of $\OSf(1)$
give a module for that algebra.  Whenever 
$S=\Spec R$ and
$f$ is not a zero-divisor,
this algebra and module are the $(C,M)$ given locally by the Key Construction.
When, for example, $f=0$, the global functions of $\Sf$ are
$\OS$, which is not a quadratic $\OS$-algebra.  In order to extend these simple
and natural constructions of $C$ and $M$ to forms $f$ that may be zero or zero divisors,
we use cohomology.  In the case of binary cubic forms, such a construction
of the ring $C$ was given by Deligne in a letter to Gan, Gross, and Savin \cite{Delcubic}.  

\begin{globcon}
 Let $\pi : \P(V) \ra S$.
We can construct
\begin{equation}\label{E:C}
C:=H^0R\pi_* \left(\mathcal{O}(-2)\tensor \pi^*L^* \stackrel{f}{\ra} \mathcal{O}\right),
\end{equation}
and 
\begin{equation}\label{E:M}
M:=H^0R\pi_* \left(\mathcal{O}(-1)\tensor \pi^*L^* \stackrel{f}{\ra} \mathcal{O}(1)\right).
\end{equation}
Above we are taking the hypercohomology of complexes with terms in degrees -1 and 0.
\emph{When $f$ is not a zero-divisor, Equations~\eqref{E:C} and \eqref{E:M} just say that $C$ is the pushforward of the functions of
$S_f$, the subscheme of $\P(V)$ cut out by $f$, to $S$.  Also, in this case, $M$ is the pushforward
of $\OSf(1)$ to $S$.}  This is because when 
 the map $\mathcal{O}(k)\tensor \pi^*L^* \stackrel{f}{\ra} \mathcal{O}(k+2)$ is injective, 
the complex $\mathcal{O}(k)\tensor \pi^*L^* \stackrel{f}{\ra} \mathcal{O}(k+2)$ is 
chain homotopic to $\mathcal{O}(k+2) / f(\mathcal{O}(k)\tensor \pi^*L^*)\isom \mathcal{O}_\Sf(k)$ (as a chain complex in the 0th degree).
Thus when $f$ is not a zero-divisor, we have
$C\isom\pi_*(\mathcal{O}_{\Sf})$ and $M\isom\pi_*(\mathcal{O}_{\Sf}(1))$.

We see that there is an associative product on the complex $\mathcal{O}(-2)\tensor \pi^*L^* \stackrel{f}{\ra} \mathcal{O}$, as it is the Kozul complex of one form.  Explicitly, the product is
given by the action of $\mathcal{O}$ on $\mathcal{O}(-2)$, which gives $C$ the structure of a ring.
The map of $\mathcal{O}$
as a complex in degree 0 to the complex $\mathcal{O}(-2) \stackrel{f}{\ra} \mathcal{O}$
induces $\mathcal{O}_\BS\isom R^0 \pi_*(\mathcal{O}) \ra C$, which makes $C$ into an $\OS$-algebra.
The $C$-module structure on $M$ 
 is given by the following
action of the complex $\mathcal{O}(-2)\tensor \pi^*L^* \stackrel{f}{\ra} \mathcal{O}$
on the complex 
 $\mathcal{O}(-1)\tensor \pi^*L^* \stackrel{f}{\ra} \mathcal{O}(1):$
 \begin{align*}
 \mathcal{O}\tensor (\mathcal{O}(-1)\tensor \pi^*L^*) \ra \mathcal{O}(-1)\tensor \pi^*L^* \qquad \mathcal{O}\tensor (\mathcal{O}(-1)\tensor \pi^*L^*) \ra \mathcal{O}(-1)\tensor \pi^*L^*\\
 (\mathcal{O}(-2)\tensor \pi^*L^*)\tensor \mathcal{O}(1) \ra \mathcal{O}(-1)\tensor \pi^*L^*
 \qquad (\mathcal{O}(-2)\tensor \pi^*L^*)\tensor (\mathcal{O}(-1)\tensor \pi^*L^* )\ra 0,
 \end{align*}
 where all maps are the natural ones.
\end{globcon}

From the short exact sequence of complexes in degrees
-1 and 0
\begin{equation}\label{E:SES2}
\begin{CD}
 @. \mathcal{O}(k+2) \\
@. @VVV \\
\mathcal{O}(k) @>f>> \mathcal{O}(k+2) \\
@VVV @.  \\
\mathcal{O}(k) @. 
\end{CD}
\end{equation}
(where each complex is on a horizontal line and $k=-1$ or $-2$),
we can apply the long exact sequence of cohomology to obtain the exact sequences
\begin{align*}
 0\ra \OS \ra 
C& \ra 
(\wedge^2 V^*)\tensor L^*
\ra 0 
\\
0\ra  V \ra 
M& \ra  0.
\end{align*}
Thus we have natural isomorphisms of $\OS$-modules
\begin{equation}\label{E:lesres}
 C/\OS \isom (\wedge^2 V^*)\tensor L^* \quad \text{and}\quad V \isom M
\end{equation}
as claimed in Theorem~\ref{T:gen}. 
So $C$ and $M$ constructed here have the same $\OS$-module structure as given in the Key Construction,
and it also turns out that the algebra and $C$-module structures are also the same.

\begin{proposition}
 The constructions of $C$ and $M$ from a form $f$ given in the Global Construction
commute with base change and are the same as the Key Construction.
\end{proposition}
\begin{proof}
They key to this proof is to compute all cohomology of the pushforward 
of the complex $\mathcal{C}(k): \mathcal{O}(k-2)\tensor \pi^*L^* \stackrel{f}{\ra} \mathcal{O}(k)$ for $k=0,1$.
This can be done using the long exact sequence of cohomology from the short exact sequence
of complexes given in Equation~\eqref{E:SES2} above.  
In particular, $\mathcal{C}(k)$ 
 does not have any cohomology
in degrees other than 0.  
Since $k\leq 1$, we have that 
$H^{0}R\pi_*(\mathcal{O}(-2+k))=0$
and thus
$H^{-1}R\pi_*(\mathcal{C}(k))=0$.
Since, $k\geq -1$ we have that $H^{1}R\pi_*(\mathcal{O}(k))=0$ and thus
$H^{1}R\pi_*(\mathcal{C}(k))=0$.  Moreover, we saw above that
$H^{0}R\pi_*(\mathcal{C}(k))$ is locally free.  Thus since all $H^{i}R\pi_*(\mathcal{C}(k))$ are flat,
by \cite[Corollaire 6.9.9]{EGA3}, we have that cohomology and base change commute.
Base change respects the induced maps between cohomology sheaves
that gave $C$ and $M$ algebra and module structures, respectively, as well
as the maps in Equation~\eqref{E:lesres}.

As in the Key Construction of $(C,M)$ from a form, this construction lifts,
using Equation~\eqref{E:lesres},
to a map $\mathcal{F}_B\ra\mathcal{M}_B$ of the rigidified moduli spaces.
This is because Equation~\eqref{E:lesres} gives a basis for $C/\OS$ and $M$ from a basis of
$V$ and $L$.
So, it suffices to check that 
on the universal linear
binary quadratic form
this construction gives the universal pair $(C,M)$.  This is proven more generally
in an analogue for binary forms of degree $n$ in \cite[Theorem 2.4]{binarynic}.
\end{proof}

\begin{globcon}
 We now give a global construction of a linear binary quadratic form from quadratic ring and module.
Let $C$ be a quadratic $\OS$-algebra and $M$ be a traceable $C$-module.
There is a natural map
$$\map{C/\OS \tensor \wedge^2_\OS M}{ \Sym^2 M}{\gamma\tensor m_1 \wedge m_2}{\gamma m_1\tensor m_2 -\gamma m_2\tensor m_1}.$$
We define $V=M$ and $L= (C/\OS \tensor \wedge^2_\OS M)^*$ and then the map above gives us an element $f\in \Sym^2 V \tensor L$.
\end{globcon}

\begin{remark}
We can rewrite this construction as 
$$\map{C/\OS \tensor \Sym_2 M }{ \wedge^2_\OS M}{\gamma\tensor m_1  m_2}{\gamma m_1\wedge  m_2 }.$$
Using the isomorphism $\Sym_2 M^* \tesnor \wt M \isom \Sym_2 M \tensor \wt M^*$, this gives a binary form
of the required form, which one can check is equivalent to the one given above.

\end{remark}

This construction clearly commutes with base change.  
Also, it gives a map $\mathcal{F}_B\ra\mathcal{M}_B$ of the moduli
space of forms with $V$ and $L$ free with chosen basis to the moduli space
of quadratic algebras and traceable modules with $C$ and $M$ free with chosen basis. 
We can easily check that on the universal $(C,M)$, this construction
gives the universal linear binary quadratic form, and thus is inverse to the Key Construction.  

\section{Ideals and Modules}\label{S:idmod}

We now relate traceable modules to ideals.
 Recall that an ideal $I$ of a quadratic $\OS$-algebra $C$ is \emph{full} if it is a locally
free rank 2 $\OS$-module.  Two ideals $I$ and $I'$ are \emph{equivalent} if there are non-zero-divisors $c,c'\in C$ such that
$cI=c'I'$, and this equivalence defines ideal classes.
Over an integral domain or scheme, an object is \emph{degenerate} if its discriminant is zero.
We prove the following proposition given in the introduction.

\begin{proposition}[Proposition \ref{P:idmod}]
 When $D$ is an integral domain or scheme, 
all full ideals of a quadratic algebra $C$ over $D$ are traceable $C$-modules.
When $C$ is non-degenerate, all
traceable $C$-modules are realized as full ideals of $C$.  Two full ideals of
$C$ are in the same ideal class if and only if they are isomorphic as modules.  
\end{proposition}

\begin{proof}
Let $K$ be the fraction field of $D$.  We have that $C\tesnor_D K$ is a 2 dimensional $K$-algebra, generated
by $1$ and $\tau$.  If $M$ is a traceable $C$-module, then $M\tensor_D K$ is a 2 dimensional $K$ vector space and
a $C\tesnor_D K$-module.  We can put the action of $\tau$ on $M\tensor_D K$ into rational normal form.  An easy calculation
shows that if $C$ is non-degenerate then $\tau$ does not have any repeated eigenvalues, and so we can assume that
$\tau$ acts on $M\tensor_D K$ by 
$\left(
\begin{smallmatrix}
0 & r\\
1 &q 
\end{smallmatrix}\right)$
and $\tau^2=-s\tau-t$ with $q,r,s,t\in K$.
Since $M$ is traceable, we have $q=-s$ and since $\tau^2$ must
act on $M$ in the same way as $-s\tau-t$ we have $r=-t$.  
Thus, as
$(C\tesnor_D K)$-modules, $M\tensor_D K\isom C\tesnor_D K$.  Since $M\subset M\tesnor_D K$,
we have realized $M$ as a $C$-submodule of $C\tesnor_D K$.  Since $M$ is finitely generated as a $C$-module,
for some non-zero $d\in D$, we have $dM\subset C$.  This realizes $M$ as a full ideal of $C$.

Let $I$ be a full ideal of $C$.  By definition $I$ is a locally free rank 2 $D$-module.  
We have that $I\tensor_D K \subset C \tesnor_D K $, and since both are two dimensional $K$ vector spaces,
we must have equality.  So $I\tensor_D K$ and $C \tesnor_D K $ give the same trace map
from $C \tesnor_D K $ to $K$, which when restricted to $C$ gives the trace maps
that $I$ and $C$ give from $C$ to $D$.  Thus $I$ is traceable.  

Clearly two ideals in the same ideal class are isomorphic as modules.  Suppose we have a module isomorphism of two full ideals
$\phi : I \ra J$.  Since $I \tesnor_D K\isom C\tesnor_D K$, there is some nonzero element $d\in D\subset C$
such that $d\in I$.  We claim that as subsets of $C$, $\phi(d)I=dJ$.  Suppose
we have an element $\phi(d) x \in \phi(d)I$ with $x\in I$.  We have $\phi(d)x=\phi(dx)=d\phi(x)$ since
$\phi$ is a $C$-module homomorphism.  Thus, $\phi(d)x\in dJ$.  Similarly we see that $dJ \subset \phi(d)I$.
Thus $I$ and $J$ are in the same ideal class.  
\end{proof}

This allows us to deduce Corollary~\ref{C:ideals}, which presents the bijection of Gauss composition
in terms of ideals instead of traceable modules.  If we further require  our base $D$ to be a Dedekind domain and that $C$ be a domain,
then all non-zero ideals of $C$ are full.  When $D$ is a domain with fraction field $K$, it is easy to check
 that
 $C$ is a domain if and only if $f$ is irreducible over $K$.  Thus we deduce the following corollary.

\begin{corollary}\label{C:ideals2}
When $D$ is a Dedekind domain with fraction field $K$, there is a discriminant preserving bijection
$$
 \left\{\parbox{2.5 in}{isomorphism classes of non-degenerate linear binary quadratic 
forms/$D$ that are irreducible over $K$}\right\}\longleftrightarrow \left\{\parbox{2.5 in}{isomorphism classes of $(C,I)$, with $C$ a non-degenerate quadratic algebra$/D$ and a domain, and
$I$ a non-zero ideal class of $C$ }\right\}.$$
\end{corollary}

\subsection{Theorem~\ref{T:gen} over \Z}
In Theorem~\ref{T:gen}, we remove the condition of non-degeneracy seen in the classical theorems in order to give the most complete theorem.  Over $\Z$, there is only one degenerate 
quadratic ring, but
over an arbitrary base, especially one with 0-divisors, much more can occur in the degenerate locus.  Moreover, one may
want to study quadratic rings and their ideal classes under base change (even for example from $\Z$ to $\Z/p\Z$), in which case non-degenerate objects may become degenerate.
Over $\Z$, however, when we include the degenerate case,
 we are already forced to include a module which is not realized as an ideal.  We give here 
Theorem~\ref{T:gen} and its
 specialization, Theorem~\ref{T:tbqf},
 which will be proven in the next section, when considered over $\Z$. 

 \begin{theorem}\label{T:Zdeg}
There are discriminant preserving bijections 
$$\bij{twisted \GZ equivalence classes of binary quadratic forms over \Z}{isomorphism classes of $(C,M)$, with $C$ a oriented quadratic ring/\Z, and $M$ a $C$-module that is a free rank 2 $\Z$-module such that
$C$ and $M$ give the same trace map $C \ra R$}$$
and
$$\bij{$\Gto$ equivalence classes of binary quadratic forms over \Z}{isomorphism classes of $(C,M)$, with $C$ a quadratic ring over $\Z$, and $M$ a $C$-module which is a free rank 2 $\Z$-module such that
$C$ and $M$ give the same trace map $C \ra R$}.$$
\end{theorem} 

The $0$ form corresponds to the ring $\Z[\tau]/\tau^2$ and the module
$\Z x\oplus \Z y$ where $\tau$ annihilates $x$ and $y$.  This module 
cannot be realized as an ideal of $\Z[\tau]/\tau^2$, and so even when the base is $\Z$ we see that we
must consider modules and not just ideals to get the most complete theorem.

\section{Other kinds of binary quadratic forms}\label{S:bqf}

We can specialize Theorem~\ref{T:gen} by specifying the locally free rank 1 module $L$.  
The linear binary quadratic forms with a given $L$ correspond to $(C,M)$ as above with
$$
C/\OS \isom \wedge^2_\OS M^* \tensor L^*
$$
as $\OS$-modules.
This can be thought us as an $L$-type orientation of $(C,M)$.   
For example,
we can fix $L=\OS$ to get binary quadratic forms (as defined in Section~\ref{S:intro}), the analog of binary
quadratic forms over \Z up to non-twisted \GZ-equivalence.  

\begin{theorem}\label{T:bqf}
There is a discriminant preserving bijection
$$
 \left\{\parbox{2.5 in}{isomorphism classes of binary quadratic 
forms/$S$}\right\}\longleftrightarrow  \left\{\parbox{2.5 in}{isomorphism classes of $(C,M)$, with $C$ a quadratic algebra$/S$, 
$M$ a traceable $C$-module, and 
$C/\OS \isom \wedge^2_\OS M^* $}\right\}.$$
Isomorphisms of $(C,M)$ are required to commute with the isomorphism $C/\OS \isom \wedge^2_\OS M^*$. 
\end{theorem}

  Another useful choice is
$L=\wedge^2 V^*$, which is not fixed but rather depends of the $V$ of the form. A \emph{twisted binary quadratic form} over $S$ is a locally free rank 2 $\OS$-module $V$
and a global section $f\in \Sym^2 V\tesnor \wedge^2 V^*$.  Isomorphisms of binary quadratic forms are
given by isomorphisms $V\ra V'$ that preserve the section.   This allows us to specialize to Theorem~\ref{T:classGL2} because
isomorphism classes of twisted binary quadratic forms over $\Z$ correspond exactly to the twisted \GZ equivalence classes
of binary quadratic forms of Theorem~\ref{T:classGL2}.
\begin{theorem}\label{T:tbqf}
There is a discriminant preserving bijection
$$
 \left\{\parbox{2.5 in}{isomorphism classes of twisted binary quadratic 
forms/$S$}\right\} \longleftrightarrow \left\{\parbox{2.5 in}{isomorphism classes of $(C,M)$, with $C$ a quadratic algebra$/S$, 
$M$ a traceable $C$-module, and 
$C/\OS \isom \OS $}\right\}.$$
Isomorphisms of $(C,M)$ are required to commute with the isomorphism $C/\OS \isom \OS$. 
\end{theorem}
The isomorphism $C/\OS \isom \OS $ is an orientation of $C$, and over $\Z$ is exactly the orientation we defined above.
So when $S=\Spec \Z$ Theorem~\ref{T:tbqf} gives us the first bijection of Theorem~\ref{T:Zdeg}, and when we 
consider
only non-degenerate objects we recover the classical Theorem~\ref{T:classGL2}.
Of course, we could get similar theorems by choosing some other $L$, either fixed or as a function of $V$, such
as $(\wedge^2 V)^{\tensor k}$.  

The proofs of Theorems~\ref{T:bqf} and \ref{T:tbqf} are completely analogous to that of Theorem~\ref{T:gen}.  Moreover,
we have global descriptions of the bijections which can be read off from the Global Constructions
in Section~\ref{S:global}.
For completeness, we give the moduli stack version of the proofs of the above theorems.

\begin{theorem}\label{T:stacksbqf}
 There is an equivalence of moduli stacks between the moduli stack of binary quadratic forms on $S$
and the moduli stack of pairs $(C,M)$ where $C$ is a quadratic $\OS$-algebra, $M$ is a traceable $C$-module,
and $C/\OS\isom \wedge^2_{\OS} M^*$ is given.
\end{theorem}
\begin{proof}
We can rigidify the two moduli problems of the theorem. 
 We have that $\A^3$ is the moduli scheme $\mathcal{F'}_B$ of binary quadratic forms $(V,f)$ on $S$ with 
given $V\isom\OS^2$.  Also, from Section~\ref{S:proof} we know that
$\A^3$ is the moduli scheme $\mathcal{M'}_B$ of
pairs $(C,M)$ where $C$ is a quadratic $\OS$-algebra, $M$ is a traceable $C$-module, and we have given isomorphisms
$C/\OS \isom \wedge^2_{\OS} M^*$ and $M \isom \OS^2$.    
An element of $g\in \GL_2$ acts on 
the isomorphisms of $M$ and $V$ with $\OS^2$  by composing with $g$
and acts on the isomorphism $C/\OS \isom \OS$ by composing with $\det g^{-1}$.
This gives an action
of the group scheme $\GL_2$ on both of these rigidifies moduli spaces.
  The Key Construction of $(C,M)$ from the universal form gives us
$\mathcal{F'}_B\isom\mathcal{M'}_B$ which is equivariant for the $\GL_2$ actions.
Thus we have an equivalence of the quotient stacks by $\GL_2$, which are the moduli stacks in this theorem.
\end{proof}

\begin{theorem}\label{T:stackstbqf}
 There is an equivalence of moduli stacks between the moduli stack of twisted binary quadratic forms on $S$
and the moduli stack of pairs $(C,M)$ where $C$ is a quadratic $\OS$-algebra, $M$ is a traceable $C$-module,
and $C/\OS\isom \OS$ is given.
\end{theorem}
\begin{proof}
We can rigidify the two moduli problems of the theorem. 
 We have that $\A^3$ is the moduli scheme $\mathcal{F''}_B$ of twisted binary quadratic forms $(V,f)$ of $S$ with 
given $V\isom\OS^2$.  Also, from Section~\ref{S:proof} we know that
$\A^3$ is the moduli scheme $\mathcal{M''}_B$ of
pairs $(C,M)$ where $C$ is a quadratic $\OS$-algebra, $M$ is a traceable $C$-module, and we have given isomorphisms
$C/\OS \isom \OS$ and $M \isom \OS^2$.  An element of $g\in \GL_2$ acts on 
the isomorphisms of $M$ and $V$ with $\OS^2$  by composing with $g$,
which gives an action of the group scheme $\GL_2$  on both of these rigidified moduli spaces.
  The construction of $(C,M)$ from the universal form gives us
$\mathcal{F''}_B\isom\mathcal{M''}_B$ which is equivariant for the $\GL_2$ actions.
Thus we have an equivalence of the quotient stacks by $\GL_2$, which are the moduli stacks in this theorem.
\end{proof}

\section{Relationship to work of Kneser}\label{S:Kneser}

In this section, we relate the work of this paper to the work of Kneser \cite{Kneser} on Gauss composition
over an arbitrary base,  First we reconcile our terminology with his.  Kneser works over an arbitrary ring $R$,
so throughout this section our base will always be a ring $R$.

Kneser works with \emph{quadratic maps} $q: M \ra N$, i.e set maps from $M$, a locally free rank 2 $R$-module, to
$N$, a locally free rank 1 $R$-module, such that for all $r\in R$ and $m\in M$, we have
$q(rm)=r^2q(m)$ and $q(x+y)-q(x)-q(y)$ is a bilinear form on $M\times M$.  

\begin{proposition}\label{P:quad}
Quadratic maps $q : M \ra N$ in the sense of Kneser described above are in bijection with linear binary quadratic forms
$f \in \Sym^2 M^* \tesnor N$, where $M$ and $N$ are $R$-modules locally free of ranks 2 and 1 respectively.  
\end{proposition}
\begin{proof}
Given an  $f \in \Sym^2 M \tesnor N$, we naturally get a homomorphism from
$(\Sym^2 M^*)^*\isom \Sym_2 M$ to $N$ which we call $Q$.  Given $m\in M$ we
can define $q(m)=Q(m\tesnor m)$.  We see that for $r\in R$,
we have $q(rm)=Q(rm\tesnor rm)=r^2 q(m)$.  Also, $q(x+y)-q(x)-q(y)=Q(xy+yx)$ which is bilinear in $x$ and $y$.
Thus $q$ is a quadratic map in Kneser's sense.

Now suppose we have a quadratic map $q : M \ra N$ in the sense of Kneser.  First assume that $M$ is
a free $R$-module generated by $m_1$ and $m_2$.  Then we know that
$q(r_1m_1 +r_2m_2)=r_1^2 q(m_1)=r_2^2q(m_2)+r_1r_2B(m_1,m_2)$, where
$B(x,y)=q(x+y) -q(x)-q(y)$ and $r_i\in R$.  We can give a map 
$\Sym_2 M \ra N$ by sending 
$m_i \tesnor m_i$ to $r_i^2(q(m_i)$
and $m_1\tesnor m_2 +m_2\tensor m_1$ to  $B(m_1,m_2)$.  It is easily checked
that
 the map $\Sym_2 M \ra N$ does not depend on the choice of basis of $M$.
Now if $M$ is a locally free $R$-module, this defines a map $\Sym_2 M \ra N$
by defining it on local patches where $M$ is free.  

An easy computation for free $M$ shows that these two constructions are inverses locally on $R$
and thus inverses.  
\end{proof}

One advantage of the $\Sym^2 M \tesnor N$ point of view that we take in this paper is that it makes it clearer
how to base change a form.  

Kneser says that a quadratic map is \emph{primitive} if $q(M)$ generates $N$ as an $R$-module.  (Kneser actually only
gives this definition for $N=R$.)  We can see that primitivity of a quadratic map is a local condition on $R$.
When $M$ is free with basis $m_1,m_2$, then $q(c_1m_1+c_2m_2)=q(m_1)c_1^2+(q(m_1+m_2)-q(m_1)-q(m_2))c_1c_2+q(m_2)c_2^2$.
If $q(M)$ generates $N$ then $q(m_1),q(m_1+m_2)-q(m_1)-q(m_2),q(m_2)$ must generate $N$ as an $R$-module.
  Conversely, if $q(m_1),q(m_1+m_2)-q(m_1)-q(m_2),q(m_2)$ generate $N$, then $q(m_1),q(m_1+m_2),q(m_2)$ and thus
$q(M)$ generate $N$ as an $R$-module.  
The corresponding
linear binary quadratic form (of Proposition~\ref{P:quad}) is primitive
if and only if $q(m_1),q(m_1+m_2)-q(m_1)-q(m_2),q(m_2)$ generate $N$, which is also a local condition on $R$.
Thus we conclude the following.

\begin{proposition}
 A quadratic map in the sense of Kneser is primitive if and only if the corresponding linear binary quadratic form is primitive.
\end{proposition}


Kneser works for most of his paper with quadratic maps $q: M \ra R$, which correspond to our binary quadratic forms $f \in \Sym^2 V$.
Kneser gives a global algebraic construction, using Clifford algebras, of a quadratic $R$-algebra $C$ and a $C$-module $M$
from a quadratic map  $q: M \ra R$.    He sees that primitive maps give invertible $C$-modules.  He 
shows
this function from quadratic maps to pairs $(C,M)$ is neither injective nor surjective and he finds the structure
of the kernel and image.  In a talk \cite{Lenstra}, Lenstra suggested that Kneser's construction could be used
to given a theorem about quadratic maps along the lines of Theorem~\ref{T:bqf} restricted to primitive non-degenerate forms, where the traceable
condition does not apepar.

Kneser further gives a global algebraic construction 
of quadratic maps (corresponding to our linear binary quadratic forms)
from $(C,M)$ with $M$ an invertible $C$-module.   Lenstra \cite{Lenstra}
suggested an algebraic Clifford algebra construction of $(C,M)$ from a quadratic map which should provide an inverse construction
and thus suggested a theorem along the lines of Theorem~\ref{T:gen} restricted to primitive forms.  Lenstra also gave a construction
of a quadratic map from $(C,M)$ similar to ours in Section~\ref{S:global}.

In this paper we remove the conditions
of primitivity and non-degeneracy to give a bijection for all linear binary quadratic 
forms which requires us to consider non-invertible ideals and introduce the idea of traceable modules.  We give a geometric description
of the construction of $(C,M)$.  We provide a new framework
of linear binary quadratic forms as elements of $\Sym^2 V \tesnor L$ which later \cite{binarynic} allows us
to study
binary forms of degree $n$.
Finally, we give proofs of theorems of Gauss composition over an arbitrary base.

\section{Acknowledgements}
The author would like to thank Manjul Bhargava for asking the questions that inspired this research, guidance along the way, and helpful feedback both on the ideas
and the exposition in this paper.  She would also like to thank Hendrik Lenstra for suggesting that there could be such general statements of Gauss composition, and to thank Chris Skinner and Lenny Taelman for suggestions for improvements to the paper.  This work was done as part of the author's Ph.D. thesis at Princeton University, and during the work she was supported by an NSF Graduate Fellowship, an NDSEG Fellowship, an AAUW Dissertation Fellowship, and a Josephine De K\'{a}rm\'{a}n Fellowship.  This paper was prepared for submission while the author was supported by an American Institute of Mathematics Five-Year Fellowship.


\begin{thebibliography}{99} 

\bibitem{B} N. Bourbaki, {\it Commutative algebra. Chapters 1--7}, Translated from the French, Reprint of the 1989 English translation, Springer, Berlin, 1998.

\bibitem{Butts}
Butts, Hubert S.; Pall, Gordon. Modules and binary quadratic forms.  Acta Arith.  15  1968 23--44.

\bibitem{Dulin}
Dulin, Bill J.; Butts, H. S. Composition of binary quadratic forms over integral domains.  Acta Arith.  20  (1972), 223--251. 


\bibitem{Cohen} H. Cohen, {\it A course in computational algebraic number theory}, Springer, Berlin, 1993.


\bibitem{Delcubic} P. Deligne, letter to  W. T. Gan, B. Gross\ and\ G. Savin, November 13, 2000.

\bibitem{Gauss}
C. F. Gauss, {\it Disquisitiones Arithmeticae}, 1801.

\bibitem{EGA3}A. Grothendieck, \'El\'ements de g\'eom\'etrie alg\'ebrique. III. \'Etude cohomologique des faisceaux coh\'erents. II, Inst. Hautes \'Etudes Sci. Publ. Math. No. 17 (1963), 91 pp.


\bibitem{Kap}
Kaplansky, Irving. Composition of binary quadratic forms.  Studia Math.  31  1968 523--530.

\bibitem{Kneser} Kneser, Martin. Composition of binary quadratic forms.  J. Number Theory  15  (1982), no. 3, 406--413.

\bibitem{Lenstra} Lenstra, Hendrik \emph{Gauss composition over an arbitrary commutative ring,}
talk at Stieltjes Onderwijsweek:
Rings of Low Rank, Lorentz Center, June 2006. 

\bibitem{Towber}
Towber, Jacob.
Composition of oriented binary quadratic form-classes over commutative rings.
Adv. in Math. 36 (1980), no. 1, 1--107. 

\bibitem{binarynic} Wood, Melanie Matchett, Rings and ideal parametrized by binary $n$-ic forms, \emph{Moduli Spaces for Rings and Ideals}, Ph.D. thesis, Princeton University, Chapter 3,
also submitted for publication.

\end{thebibliography}
\end{document}